\documentclass[a4paper,10pt]{amsart}
\usepackage[english]{babel}
\usepackage[utf8]{inputenc}
\usepackage[T1]{fontenc}
\usepackage{csquotes}
\usepackage[style=numeric,
	useprefix,%
	giveninits=true,%
	hyperref,%
	doi=false,%
	url=false,%
	isbn=false,%
	backend=bibtex,%
	maxbibnames=99%
	]{biblatex}
\bibliography{./BIB}

\usepackage{amssymb}
\usepackage{mathrsfs}
\usepackage{hyperref}
\usepackage[usenames,dvipsnames]{xcolor}
\hypersetup{colorlinks,%
citecolor=Black,%
filecolor=Black,%
linkcolor=Black,%
urlcolor=Black}
\usepackage{enumitem}	
\usepackage{mathtools}	
\mathtoolsset{showonlyrefs=true}
\allowdisplaybreaks 
\newcommand{\scr}[1]{\mathscr{#1}}
\newcommand{\frk}[1]{\mathfrak{#1}}

\newcommand{\R}{\mathbb{R}}	
\newcommand{\Id}{\mathrm{Id}}	
\newcommand{\dd}{\,\mathrm{d}}	
\newcommand{\de}{\partial}		
\renewcommand{\div}{\operatorname{div}}	
\newcommand{\THEN}{\Rightarrow}	

\newcommand{\ad}{\mathrm{ad}}

\newcommand{\Aut}{\mathtt{Aut}}

\newcommand{\Ad}{\operatorname{Ad}}

\newcommand{\grad}{\nabla}
\newcommand{\vol}{\mathrm{vol}}
\newcommand{\laplacian}{\triangle}
\newcommand{\trace}{\mathrm{trace}}

\newcommand{\Lie}{\mathrm{Lie}}
\newcommand{\heis}{\mathfrak{h}}
\newcommand{\Heis}{\mathbb{H}}

\usepackage{dsfont}

   \def\XXint#1#2#3{{\setbox0=\hbox{$#1{#2#3}{\int}$}
        \vcenter{\hbox{$#2#3$}}\kern-.5\wd0}}

\theoremstyle{plain}
\newtheorem{proposition}{Proposition}[section]
\newtheorem{theorem}[proposition]{Theorem}
\newtheorem{lemma}[proposition]{Lemma}
\newtheorem{corollary}[proposition]{Corollary}
\newtheorem{thm}{Theorem}[section]

\theoremstyle{definition}

\newtheorem{remark}[proposition]{Remark}

\theoremstyle{remark}

\title[Equivalence of sub-Laplacian on Polarized groups]{Equivalence of sub-Laplacians\\ on Polarized groups}
\author[Kijowski]{Antoni Kijowski}
\address[Kijowski]{Montec LLC, Rydygiera 8/6, 01-793, Warsaw, Poland}
\email{antekijowski@gmail.com}
\author[Nicolussi Golo]{Sebastiano Nicolussi Golo}
\address[Nicolussi Golo]{University of Fribourg, Chemin du Musée 23, 1700 Fribourg, Switzerland}
\email{sebastiano2.72@gmail.com}
\author[Warhurst]{Ben Warhurst}
\address[Warhurst]{Institute of Mathematics, University of Warsaw, ul. Banacha 2, 02-097 Warsaw, Poland}
\email{b.warhurst@mimuw.edu.pl}

\date{\today}

\subjclass[2020]{
35B06, 
53C17, 
35H20, 
53C30, 
22F30, 
22E25.} 

\keywords{sub-Laplacian, symmetries of PDEs, sub-Riemannian Lie group, Carnot group, Heisenberg group.}

\begin{document}
\maketitle

\begin{abstract}
	We characterize smooth maps between sub-Riemannian Lie groups that commute with sub-Laplacians.
	We show they are sub-Riemannian conformal submersions.
	Our work clarifies the analysis initiated on Carnot groups in \cite{MR2363343}.
	In particular, we show that the sub-Laplacian in a Carnot group determines the sub-Riemannian structure.
\end{abstract}

\setcounter{tocdepth}{2}
\phantomsection
\addcontentsline{toc}{section}{Contents}
\tableofcontents

\section{Introduction}

Laplace operators, defined as $\laplacian u :=\div(\grad u)$, depend on the choice of a metric and a measure.
The measure determines the divergence $\div$ of vector fields.
The metric, Riemannian or sub-Riemannian, determines the gradient $\grad u$ (a vector field) in terms of the differential $\dd u$ (a differential form).

On a Lie group, we assume both the metric and the measure to be left invariant. 
These conditions restrict the class of Laplace operators we obtain.
For example, on the abelian Lie group $\R^n$, there exists just one Laplacian up to a change of coordinates.
In other words, if $\laplacian_1$ and $\laplacian_2$ are two Laplace operators determined by different choices of scalar product and Haar measure on $\R^n$, then there exists a linear isomorphism $F:\R^n\to\R^n$ such that, for every $u\in C^2(\R^n)$, $\laplacian_2(u\circ F) = (\laplacian_1u)\circ F$.

We are firstly interested in similar results for Carnot groups or, more generally, sub-Riemannian Lie groups.
We will describe sub-Riemannian Lie groups as tuples $(G,V,\langle \cdot,\cdot \rangle,\vol_G)$,
where $G$ is a Lie group with Lie algebra $\frk g$,
$V=V(G)\subset \frk g$ is a vector space that is bracket generating in $\frk g$,
$\langle \cdot,\cdot \rangle$ is a scalar product on $V$,
and $\vol_G$ is a left-invarian Haar measure on $G$.
See Section~\ref{sssec:num2.2} for details.
Typical examples of sub-Riemannian Lie groups are Riemannian Lie groups (when $V=\frk g$) and Carnot groups (when $V$ is the first layer of a stratification).

These data determine a (sub-)Laplace operator $\laplacian_G u = \div_G(\grad_G u)$, 
where $\div_G$ is the divergence defined by $\vol_G$ and $\grad_Gu$ is the horizontal gradient of $u$ determined by $(V,\langle \cdot,\cdot \rangle)$.
The operator $\laplacian_G$ is sub-elliptic, but not elliptic as soon as $V\neq\frk g$.

The study of maps that commute with a Laplace operator is not new.
Helgason proved in \cite[p.~387]{MR1834454} that a diffeomorphism $F:G\to H$ between Riemannian manifolds is an isometry if and only if it is a Laplacian commuting map, i.e., $\laplacian_G(u\circ F) = (\laplacian_H u)\circ F$ for all $u$ of class $C^2$ on $H$.
Watson extended this result in~\cite{MR365419}: a smooth map $F:G\to H$ is a Laplacian commuting map if and only if it is a harmonic Riemannian submersion, that is, an orthogonal submersion that is critical for a Dirichlet energy.
Goldberg and Ishihara gave a further generalization in \cite{MR520606}.
Fuglede~\cite{MR499588} and Ishihara~\cite{MR545705} independently showed that a smooth map $F:G\to H$ between Riemannian manifolds is a Laplacian commuting map if and only if is is a harmonic morphism, that is, it pulls back harmonic functions to harmonic functions.
See also~\cite{MR1068448,MR1630785,MR2414232,MR3492899}, and further references in the bibliography of harmonic morphisms \cite{Thebibliographyofharmonicmorphisms}.

We have a slightly different approach.
If $F:G\to H$ is a smooth map between sub-Riemannian Lie groups, then the operator $u\mapsto \laplacian_G(u\circ F)$ is a differential operator $C^2(H)\to C^0(G)$ of order two.
Thus, we can write $\laplacian_G(u\circ F) = P_2u + P_1u + P_0u$, where each $P_j:C^2(H)\to C^0(G)$ is a homogeneous differential operator of order $j$.
The following question arises: 
{\it Is $P_2$ of the form $P_2u = \lambda^2\cdot (\laplacian_Hu)\circ F$, for some $\lambda:G\to(0,+\infty)$?}
We show that this happens exactly when $F$ is a conformal submersion, in which case the lower order terms $P_1$ and $P_0$ are determined.
A smooth map $F:G\to H$ between sub-Riemannian Lie groups is a \emph{conformal submersion} of factor $\lambda$ if, for every $p$, the adjoint map between the dual spaces
\[
(DF(p)|_{V(G)})^* : V(H)^* \to V(G)^*
\]
is a homothetic embedding of factor $\lambda(p)$.
See~Section~\ref{sec670d0abe}.

\begin{thm}\label{thm6763089e}
	Let $G$ and $H$ be sub-Riemannian Lie groups, with $\Omega_G\subset G$ and $\Omega_H\subset H$ open. 
	Suppose that $F:\Omega_G\to \Omega_H$ is a $C^2$-smooth map,
	and that $\lambda:\Omega\to[0,+\infty)$,
	$b:\Omega\to V(H)$,
	and $c:\Omega\to\R$ are continuous functions.

	The following statements are equivalent:
	\begin{enumerate}[label=(\roman*)]
	\item\label{thm6763089e_1}
	for all $u\in C^2(\Omega_H)$,
	\[
	\laplacian_G(u\circ F) = \lambda^2 \cdot (\laplacian_Hu)\circ F 
	+ \langle b , (\grad_Hu)\circ F \rangle_H + c \cdot (u\circ F).
	\]
	\item\label{thm6763089e_2}
	$F$ is a conformal submersion of factor $\lambda$, $c\equiv0$ and
	\[
	b(p) =  \trace_G(D^2F(p)) - DF(p)[\grad_G(\log\mu_G)(p)] + \lambda(p)^2\grad_H(\log\mu_H)(F(p)) .
	\]
	\end{enumerate}
	In the case $\dim(G)=\dim(H)$ and $\lambda>0$, then both conditions are equivalent to $F$ being a conformal $C^2$ diffeomorphism.
	In particular, if $\lambda$ is constant then $F$ is a homothety;
	if $\lambda\equiv 1$, then $F$ is an isometry.
\end{thm}
\begin{proof}
	Theorem~\ref{thm67082cc6} in section~\ref{sec67082cbe} proves the implication $\ref{thm6763089e_2}\THEN\ref{thm6763089e_1}$
	
	For the implication $\ref{thm6763089e_1}\THEN\ref{thm6763089e_2}$, 
	Theorem~\ref{thm670e8117} in Section~\ref{sec670d7088} proves that a function $F$ satisfying even the apparently weaker conditions~\eqref{eq670e8147}, or~\eqref{eq670d7216}, is necessarily a conformal submersion of factor $\lambda$.
	The form of the vector $b$ is then given again by Theorem~\ref{thm67082cc6}.

	Finally, we observe that conformal submersions between groups with the same (topological) dimension are diffeomorphisms and fall into the definition of conformal maps.
\end{proof}

When $\lambda\equiv1$, Theorem~\ref{thm6763089e} gives a characterization of submetries of class $C^2$ in terms of the sub-Laplacian.
When $\lambda\equiv1$ and $\dim(G)=\dim(H)$, Theorem~\ref{thm6763089e} gives a characterization of isometries in terms of the sub-Laplacian.

In the particular case of Carnot groups we have the following corollary.
A \emph{Carnot group} is a connected simply connected sub-Riemannian Lie group whose polarization $V(G)\subset\frk g$ is the first layer of a stratification of $\frk g$.
See~\cite{MR3742567,donne2024metricliegroupscarnotcaratheodory} for details.
The important feature of Carnot groups is that they support dilations, that is, homotheties for all positive factors $\lambda>0$.

\begin{thm}\label{thm670e8a84}
	Let $G$ and $H$ be sub-Riemannian Carnot groups.
	If there exist a smooth map $F:G\to H$ of class $C^2$ and $\lambda>0$
	such that, for all $u\in C^2(H)$,
	\begin{equation}\label{eq676b2756}
	\laplacian_G(u\circ F) = \lambda^2 (\laplacian_Hu)\circ F
	\end{equation}
	then $H$ is a quotient of $G$, as Carnot groups.
	
	If $\dim(G)=\dim(H)$, then $F$ itself is the composition of a dilation, a left translation, and an (isometric) isomorphism of Carnot groups. 
	In particular, $G$ and $H$ are isomorphic as Carnot groups.
\end{thm}

Theorem~\ref{thm670e8a84} is a particular case of Propositions~\ref{prop6767ac89} and~\ref{prop67682013}: see section~\ref{sec67633de5} for their proofs.

\begin{remark}
	The map $F$ in Theorem~\ref{thm670e8a84} is in fact $C^\infty$ smooth.
	Indeed, if $u\in C^\infty(H)$ and if $F$ is $C^2$, then~\eqref{eq676b2756} implies that $u\circ F$ is $C^\infty$ smooth, by a standard bootstrap argument that uses the hypoellipticity of the sub-Laplacian.
	Since $u$ is arbitrary, we can take a system of coordinates in $H$ and conclude that $F$ is $C^\infty$ smooth.
	Such an argument cannot be used under the hypothesis of Theorem~\ref{thm6763089e}.
\end{remark}

Theorem~\ref{thm670e8a84} answers to a question from \cite{MR2363343}.
In \cite[Chapter 16]{MR2363343}, it is asked whether two differential operators defined as sum of squares of vector fields on a Carnot group are equivalent.
More precisely, if $G$ is a simply connected Lie group with stratified Lie algebra $\frk g = \bigoplus_{j=1}^sV_j$, given two basis $X = \{X_1,\dots,X_r\}$ and $Y=\{Y_1,\dots,Y_r\}$ of $V_1$, 
define $\laplacian_X = \sum_{j=1}^r \tilde X_j^2$ and $\laplacian_Y = \sum_{j=1}^r \tilde Y_j^2$.
We use the tilde $\tilde v$ to denote the left-invariant vector field with $\tilde v(1_G)=v \in\frk g$.
Thus, we wonder whether there exists a diffeomorphism $F_0:G\to G$ such that $\laplacian_X(u\circ F_0) = (\laplacian_Yu)\circ F_0$ for all $u\in C^2(G)$.

In \cite[Chapter 16]{MR2363343}, it is shown that if $G$ is a free Lie group, then such a diffeomorphism $F_0$ always exists, and it is a Lie group automorphism.
It is also shown that in Heisenberg groups it may not exist.

Theorem~\ref{thm670e8a84} implies that such an diffeomorphism $F_0$ exists if and only if the two sub-Riemannian Carnot groups given by the scalar products on $V_1$ that make  $X$ and $Y$ orthonormal, respectively, are isometric.
In particular, such an isomorphism $F_0$ is itself an isometry and a Lie group automorphism.
However, $F_0$ does not need to map the basis $X$ to $Y$.
See also Corollary~\ref{cor676308cb}.

In this way, we recover the results shown in \cite[Chapter 16]{MR2363343} for free groups.
Indeed, if $G$ is a free Lie group, then every choice of two basis $X$ and $Y$ for $V_1$ determines a Lie group automorphism $F_0:G\to G$ that such that $(F_0)_* X_j= Y_j$ for all $j$.
Moreover, we give a complete description of all equivalence classes of sub-Laplacians in Heisenberg groups,
see Section~\ref{subs676e6b89}.

Finally, we notice that, since every Carnot group is a Carnot quotient of a free Carnot group, then every sub-Laplacian on a Carnot group is induced by the sub-Laplacian of a free group via a conformal submersion of factor 1, i.e., a submetry, which is the quotient map itself.

\subsection*{Acknowledgments}
Kijowski and Nicolussi Golo started this project at the {\it Okinawa Institute of Science and Technology (OIST)} in Okinawa, Japan. They are grateful for the support of OIST and, in particular, of the TVSP program in OIST, for providing a great opportunity and the perfect environment to work together.

Nicolussi Golo was partially supported 
by the Swiss National Science Foundation (grant 200021-204501 ‘{\it Regularity of sub-Riemannian geodesics and applications}’), 
by the European Research Council (ERC Starting Grant 713998 GeoMeG `{\it Geometry of Metric Groups}'), 
by the Academy of Finland 
(grant 288501 `{\it Geometry of subRiemannian groups}', 
grant 322898 `{\it Sub-Riemannian Geometry via Metric-geometry and Lie- group Theory}', 
grant 328846 `{\it Singular integrals, harmonic functions, and boundary regularity in Heisenberg groups}' 
grant 314172 `{\it Quantitative rectifiability in Euclidean and non-Euclidean spaces}'). 

The authors thank Andrea Pinamonti for pointing out an error in the formula for $b(p)$ in Theorem~\ref{thm6763089e} in an earlier version of this paper.
The authors also used Claude (Anthropic) to verify the sign of the modular-function terms in this correction.

\section{Preliminaries}

\subsection{General notations}
If $M$ is a smooth manifold, we denote by $TM$ its tangent bundle
and by $\Gamma^k(TM)$ the sections of $TM$ of class $C^k$.
For $k=\infty$, we just write $\Gamma(TM):=\Gamma^\infty(TM)$.
If $F:M\to N$ is a $C^1$ map between smooth manifolds, we denote by $dF:TM\to TN$ its standard differential-geometric differential.

\subsection{Lie groups}
Given a Lie group $G$, we identify, as vector spaces, its Lie algebra $\frk g=\Lie(G)$ with the tangent space $T_1G$ of $G$ at the identity $1_G$ of $G$.
We denote by $L_p$ the left translation by $p$.
If $v:E\to\frk g$ for some $E\subset G$, then $\tilde v:E\to TG$ is the map 
\[
\tilde v(p) := dL_p[v] \in T_pG .
\]
If $f$ is a vector valued $C^1$-function defined in a neighborhood of $p\in G$, then
\[
\tilde vf(p) := df(p)[\tilde v(p)] = \left.\frac{\dd}{\dd t}\right|_{t=0} f(p\exp(tv)) .
\]
If $v\in\frk g$, then $\tilde v$ is the left-invariant vector field with value $v$ at $1$.

\subsection{Polarized and Sub-Riemannian Lie groups}\label{sssec:num2.2}
A \emph{polarized Lie group} is a pair $(G,V)$ where $G$ is a connected Lie group and $V\subset\frk g$ is a bracket-generating subspace of the Lie algebra $\frk g$ of $G$. We call $V$ a \emph{polarization} of $G$ and denote it by $V(G)$. 
By \emph{bracket-generating subspace} we mean a linear subspace that Lie generates the Lie algebra, i.e., $\frk g$ is the smallest Lie algebra containing $V$. 

Carnot groups are the stratified nilpotent examples. 
An example where $\frk g$ is not nilpotent is $SL_2(\R)$. 
In this case the Lie algebra $\frk{sl}_2(\R)$ has the basis 
\[
	X = \begin{pmatrix} 0 & 1 \\ 0 & 0 \end{pmatrix} 
	\quad Y=\begin{pmatrix} 0 & 0 \\ 1 & 0 \end{pmatrix} 
	\quad Z=\begin{pmatrix} 1 & 0 \\ 0 & -1 \end{pmatrix} 
\] 
where 
\[
	[X,Y]=Z \quad [Z,X]=2X \quad [Z,Y]=-2Y. 
\]
A two dimensional polarization is given by $V(G)={\rm span} \{X,Y\}$. 

In general, the bundle of left translates 
\[
\tilde V(G) := \bigcup_{p \in G} d L_p(V)
\]
forms a \emph{horizontal subbundle} of $TG$, which is bracket generating. 

A \emph{sub-Riemannian Lie group} is a pair $(G,\langle \cdot,\cdot \rangle_G)$ where $G$ is a polarized Lie group and $\langle \cdot,\cdot \rangle_G$ is a scalar product on the polarization $V(G)$. 

\subsection{Lie differential of order one}
Let $G$ and $H$ be Lie groups and $\Omega\subset G$ open.
The \emph{Lie differential} of a smooth function $F:\Omega\to H$ at $p\in \Omega$ is the  map $ DF : \Omega \to \frk h \otimes \frk{g}^* $ defined by 
\begin{equation}\label{DFdef1}
	DF(p)[v] := \left.\frac{\dd}{\dd t}\right|_{t=0} F(p)^{-1}F(p\exp(tv))  .
\end{equation}
We will use square brackets $[\cdot]$ to highlight the (multi-)linear entry of a function.
In some cases we will drop the brackets and just juxtapose the vector as in $DF(p)v$.

\begin{proposition}\label{PropChainRule1} 
	The basic properties of the Lie differential $D$ are as follows:
	\begin{enumerate}[label=(\ref{PropChainRule1}.\alph*),leftmargin=*] 
	\item\label{PropChainRule1_1}
	Let $G$ and $H$ be Lie groups, $\Omega\subset G$ open,
	and $F:\Omega\to H$ a $C^1$-function.
	Then the Lie differential $DF$ is a continuous map $\Omega\to \frk h \otimes \frk g^*$ and decomposes as
	\begin{equation*}
	DF(p)[v] = dL_{F(p)}^{-1}[ dF(p)[dL_p[v]] ] , 
	\qquad \forall p\in\Omega,\ v\in\frk g,
	\end{equation*}
	where $d$ denotes the standard differential-geometric differential and $L_p$ denotes the left translation by $p$.
	\item\label{PropChainRule1_2}
	Let $G,H,K$ be Lie groups, $\Omega_G\subset G$ and $\Omega_H\subset H$ open sets.
	Let $A:\Omega_G\to H$ and $B:\Omega_H\to K$ be $C^1$-functions.
	Then, if $p\in\Omega_G$ is such that $A(p)\in\Omega_H$,
	and if $v\in\frk g$, then
	\begin{equation*}
	D(B\circ A)(p)[v] = DB(A(p)) [ DA(p)[v] ] .
	\end{equation*}
	\item\label{PropChainRule1_3}
	Let $G$ and $H$ be Lie groups. 
	For every $a,p\in G$ and every Lie group morphism $A:G\to H$,
	\begin{equation*}
	D(A\circ L_a)(p) = A_* .
	\end{equation*}
	\end{enumerate}
\end{proposition}

\subsection{Lie differential of order two}
Let $G$ and $H$ be two polarized Lie groups, $\Omega\subset G$ open
and $F:\Omega\to H$ a $C^2$ function.
Notice that the function $DF:\Omega\to\frk h\otimes\frk g^*$ is a $C^1$ function valued in the Abelian Lie group $\frk h\otimes\frk g^*$.
As such, we define the \emph{second order Lie differential} of $F$ as the Lie differential $D(DF)$ of $DF$,
that is, as the function $D^2F: \Omega \to (\frk{h} \otimes \frk g^*) \otimes \frk g^*$,
\begin{equation}\label{eq67684517}
\begin{aligned}
	D^2F(p)[v,w] 
	&:= (D(DF)(p)[w])[v] 
	= \left.\frac{\dd}{\dd t}\right|_{t=0} ( DF(p \exp(t w))[v] - DF(p)[v] )\\
	&= \left.\frac{\dd}{\dd t}\right|_{t=0} DF(p \exp(t w))[v] \\ 
	&= \left.\frac{\dd}{\dd t}\right|_{t=0} \left.\frac{\dd}{\dd s}\right|_{s=0} F(p\exp(tw))^{-1}F(p\exp(tw)\exp(sv)) . 
\end{aligned}
\end{equation}

Notice that $D^2F(p)[v,w]$ need not be symmetric in $[v,w]$.
If $u:G\to\R$ is $C^2$ and $v\in V(G)$, then~\eqref{eq67684517} gives
\begin{equation}\label{eq67683e74}
\begin{aligned}
	D^2u(p)[v,v]
	&= \left.\frac{\dd}{\dd t}\right|_{t=0} \left.\frac{\dd}{\dd s}\right|_{s=0} ( u(p\exp(tv)\exp(sv)) - u(p\exp(tv))) \\
	&= \left.\frac{\dd}{\dd t}\right|_{t=0} \left.\frac{\dd}{\dd s}\right|_{s=0} u(p\exp((t+s)v)) \\
	&= \left.\frac{\dd}{\dd t}\right|_{t=0} \tilde vu(p\exp(tv)) \\
	&= \tilde v^2u(p) .
\end{aligned}
\end{equation}

\begin{lemma}
	Let $G,H,K$ be Lie groups, $\Omega_G\subset G$ and $\Omega_H\subset H$ open sets.
	Let $A:\Omega_G\to\Omega_H$ and $B:\Omega_H\to K$ be $C^1$-functions. 
	Then, for all $p\in\Omega_G$ and $v,w\in\frk g$,
	\begin{equation}\label{eq670d7a23}
	\begin{split}
		&D^2(B\circ A)(p)[v,w] \\
		&\qquad= D^2B(A(p))[DA(p)[v],DA(p)[w]] + DB(A(p))[D^2A(p)[v,w]] .
	\end{split}
	\end{equation}
\end{lemma} 
\begin{proof}
	If $A$ and $B$ are maps satisfying the assumptions, then setting $F=B \circ A$ in \eqref{eq67684517} and applying Proposition~\ref{PropChainRule1_2} leads to the following straightforward computation
	\begin{align*}
		D^2(B\circ A)(p)[v,w]
		&\overset{\eqref{eq67684517}}= \left.\frac{\dd}{\dd t}\right|_{t=0} D(B\circ A)(p \exp(t w))[v] \\
		&\overset{\ref{PropChainRule1_2}}= \left.\frac{\dd}{\dd t}\right|_{t=0} DB(A(p\exp(tw))[DA(p\exp(tw))[v]] \\
		&= \left.\frac{\dd}{\dd t}\right|_{t=0} DB(A(p\exp(tw))[DA(p)[v]] \\
		&\qquad\qquad\qquad\qquad	+ \left.\frac{\dd}{\dd t}\right|_{t=0} DB(A(p))[DA(p\exp(tw))[v]] \\
		&\overset{\ref{PropChainRule1_2}\&\eqref{eq67684517}}= D^2B(A(p))[DA(p)[v],DA(p)[w]] \\
		&\qquad\qquad\qquad\qquad	+ DB(A(p))[D^2A(p)[v,w]] .
	\end{align*}
\end{proof}

\subsection{Contact maps}
Let $G$ and $H$ be polarized Lie groups and $\Omega\subset G$ open.
A $C^1$-map $F:\Omega\to H$ is a \emph{contact map} if 
\[
DF(p)[V(G)] \subseteq V(H),
\qquad\forall p\in\Omega.
\]

Moreover by Proposition~\ref{PropChainRule1_1}, $F$ is a contact map if and only if $dF(p)[\tilde V(G)_p] \subseteq \tilde V(H)_{F(p)}$ for all $p\in\Omega$.

Notice that, if $F:\Omega\to H$ is a contact map of class $C^2$, then its second order Lie differential 
$D^2F(p)$ at $p\in\Omega$ preserves the polarizations, in the sense that
$D^2F(p)[V(G),\frk g] \subseteq V(H)$.
This follows immediately from \eqref{eq67684517}.

If $G$ and $H$ are sub-Riemannian Lie groups, and if $F:\Omega\to H$ is a contact map of class $C^2$, then the \emph{trace} of the bilinear map $D^2F(p)|_{V(G)}:V(G)\times V(G)\to V(H)$ is 
\[
\trace_G(D^2F(p)) := \sum_{i=1}^r  D^2F(p)[X_i,X_i] ,
\]
for one (thus every) orthonormal basis $X_1,\dots,X_r$ of $V(G)$.
The fact that the definition of $\trace_G(D^2F(p))$ does not depend on the choice of the orthonormal basis of $V(G)$ is a standard computation.
Indeed, if $Y_1,\dots,Y_r$ is another orthonormal basis of $V(G)$, 
then there exists a matrix $a\in\mathrm{GL}(r)$ with $a^T = a^{-1}$ and $Y_i = a_i^j X_j$ (sum over $j$); hence
\begin{align*}
	\sum_{i=1}^r D^2F(p)[Y_i,Y_i]
	&= \sum_{i=1}^r D^2F(p)[a_i^j X_j , a_i^kX_k] 
	= \sum_{i=1}^r a_i^j a_i^k D^2F(p)[ X_j , X_k] \\
	&\overset{a_i^k = (a^T)^i_k}= \sum_{i=1}^r a_i^j (a^T)^i_k D^2F(p)[ X_j , X_k] \\
	&\overset{a^T = a^{-1}}= \sum_{i=1}^r a_i^j (a^{-1})^i_k D^2F(p)[ X_j , X_k] \\
	&= \delta_k^j D^2F(p)[ X_j , X_k] 
	= \sum_{\ell=1}^r D^2F(p)[ X_\ell , X_\ell] .
\end{align*}

\subsection{Homothetic projections}
Let $V,W$ be Hilbert spaces.
The \emph{transpose} of a linear map $L:V\to W$ is the linear map $L^T:W\to V$ such that
\[
\langle L^Tw,v \rangle_V = \langle w,Lv \rangle_W
\]
for all $v\in V$ and $w\in W$.

A linear map $L:V\to W$ is a \emph{homothetic projection} of factor $\lambda\ge0$ if
\begin{equation}\label{eq6708401d}
\langle L^Tw_1,L^Tw_2 \rangle_V  = \lambda^2 \langle w_1,w_2 \rangle_W ,
\end{equation}
for all $w_1,w_2\in W$,
i.e., the transpose $L^T:W\to V$ is a \emph{homothetic embedding} of factor $\lambda$.
Clearly, using the polarization identity, one can equivalently check~\eqref{eq6708401d} only for $w_1=w_2$.
It follows that, if $\lambda=0$ then $L=0$, while if $\lambda>0$ then $L$ is surjective.
Moreover, we have the equivalent formulations in Proposition~\ref{prop676ff612}.

\begin{lemma}\label{lem6771bbb2}
	Let $V$ and $W$ be Hilbert spaces (not necessarily finite-dimensional).
	If $L:V\to W$ is a homothetic projection of factor $\lambda>0$, then:
	\begin{enumerate}[label=(\ref{lem6771bbb2}.\alph*),leftmargin=*]
	\item\label{lem6771bbb2_1}
	$L(V)=W$;
	\item\label{lem6771bbb2_2}
	$L^T(W)=\ker(L)^\perp$;
	\item\label{lem6771bbb2_3}
	$LL^Tw = \lambda^2 w $, for every $w\in W$;
	\item\label{lem6771bbb2_4}
	$L^TLv = \lambda^2v$, for every $v\in\ker(L)^\perp$;
	\item\label{lem6771bbb2_5}
	the map $\ker(L)^\perp\to W$, $v\mapsto Lv/\lambda$, is an isometry.
	\end{enumerate}
\end{lemma}
\begin{proof}
	To show~\ref{lem6771bbb2_1}, notice that, if $w\in L(V)^\perp$, then
	\[
	\langle w,w \rangle
	= \frac1{\lambda^2}\langle L^Tw,L^Tw \rangle 
	= \frac1{\lambda^2}\langle w,LL^Tw \rangle
	= 0 .
	\]
	
	For~\ref{lem6771bbb2_2}, we have that, 
	on the one hand, if $w\in W$ then, for all $v\in \ker(L)$, 
	$\langle L^Tw,v \rangle = \langle w,Lv \rangle = 0$.
	Thus, $L^T(W)\subset\ker(L)^\perp$.
	On the other hand, if $v\in L^T(W)^\perp$, then, for all $w\in W$,
	$\langle Lv,w \rangle = \langle v,L^Tw \rangle = 0$,
	that is, $v\in\ker(L)$.
	Thus, $L^T(W)^\perp \subset\ker(L)$, and therefore $L^T(W) \supset\ker(L)^\perp$.
	
	Next, we show~\ref{lem6771bbb2_3}:
	If $w\in W$, then, for all $\bar w\in W$, we have
	$\langle LL^Tw,\bar w \rangle 
	= \langle L^Tw,L^T\bar w \rangle
	= \lambda^2 \langle w,\bar w \rangle$.
	Therefore, $LL^Tw=\lambda^2w$.
	
	Finally, we show~\ref{lem6771bbb2_4}:
	if $v\in\ker(L)^\perp$, then there is $w\in W$ such that $L^Tw=v$ and thus
	$L^TLv = L^TLL^Tw = L^T(\lambda^2w) = \lambda^2v$.
	
	The last statement~\ref{lem6771bbb2_5} is an easy consequence of the previous ones.
	Indeed, if $v,\bar v\in\ker(L)^\perp$, then
	$\langle Lv,L\bar v \rangle
	\overset{\eqref{eq6708401d}}= \lambda^{-2}\langle L^TLv,L^TL\bar v \rangle
	\overset{\ref{lem6771bbb2_4}}= \lambda^{-2}\langle \lambda^2v,\lambda^2\bar v \rangle
	= \lambda^2 \langle v,\bar v \rangle$,
	which implies~\ref{lem6771bbb2_5}.
\end{proof}

\begin{proposition}\label{prop676ff612}
	Let $L:V\to W$ be a linear map and $\lambda>0$.
	Let $\ker(L)^\perp\lhd V$ be the subspace of $V$ orthogonal to $\ker(L)$ and $\pi_{\ker(L)^\perp}:V\to \ker(L)^\perp$ the orthogonal projection.
	The following assertions are equivalent:
	\begin{enumerate}[label=(\roman*)]
	\item\label{prop676ff612_1}
	$L$ is a homothetic projection of factor $\lambda$.
	\item\label{prop676ff612_2}
	$L$ is surjective, and $L^TL = \lambda^2 \pi_{V^\perp}$.
	\item\label{prop676ff612_3}
	$LL^T = \lambda^2\Id_W$.
	\item\label{prop676ff612_4}
	The map $\ker(L)^\perp\to W$, $v\mapsto Lv/\lambda$, is an isometry.
	\item\label{prop676ff612_5}
	The adjoint map $L^*:W^*\to V^*$ between the dual spaces is a homothetic embedding of factor $\lambda$.
	\end{enumerate}
\end{proposition}
\begin{proof}
	The equivalence of \ref{prop676ff612_1} and \ref{prop676ff612_4} is a consequence of the equivalence between the transpose $L^T$ and the adjoint map $L^*$ in Hilbert spaces.
	
	Lemma~\ref{lem6771bbb2} shows that \ref{prop676ff612_1} implies \ref{prop676ff612_2}, \ref{prop676ff612_3}, and \ref{prop676ff612_4}.
	
	We will next show the remaining implications.
	
	$\ref{prop676ff612_2}\THEN\ref{prop676ff612_3}$:
	Let $w\in W$.
	Since $L$ is surjective, there is $v\in \ker(L)^\perp$ such that $Lv=w$.
	Since $L^TL = \lambda^2 \pi_{\ker(L)^\perp}$, then
	$L^Tw = L^TLv = \lambda^2 v$.
	Thus $LL^T w = \lambda^2 Lv = \lambda^2w$.
	
	$\ref{prop676ff612_3}\THEN\ref{prop676ff612_1}$:
	Let $w,\bar w\in W$.
	Then $\langle L^Tw,L^T\bar w \rangle = \langle w,LL^T\bar W \rangle = \lambda^2\langle w,\bar w \rangle$.
	
	$\ref{prop676ff612_4}\THEN\ref{prop676ff612_3}$:
	Set
	\[
	f := (L|_{\ker(L)^\perp})^{-1}:W\to\ker(L)^\perp .
	\]
	Then, for every $w\in W$ and $v\in V$,
	\[
	\langle L^Tw,v \rangle
	= \langle w,Lv \rangle
	= \langle Lf(w),L\pi_{\ker(L)^\perp}v \rangle
	\overset{\ref{prop676ff612_4}}= \lambda^2 \langle f(w),\pi_{\ker(L)^\perp}v  \rangle
	=  \langle \lambda^2 f(w), v  \rangle  .
	\]
	Therefore, $L^T = \lambda^2 f$.
	It follows that 
	$LL^Tw = L(\lambda^2 f(w)) = \lambda^2 w$, that is,~\ref{prop676ff612_3}.
\end{proof}

\subsection{Conformal submersion of sub-Riemannian Lie groups}\label{sec670d0abe}
Let $G$ and $H$ be sub-Riemannian Lie groups, and $\Omega\subset G$ open. 
A \emph{$C^1$-conformal submersion} of factor $\lambda:\Omega\to[0,+\infty)$ is a $C^1$ contact map $F:\Omega\to H$,
such that, for every $p\in\Omega$, the restriction $DF(p)|_{V(G)}:V(G)\to V(H)$ is a linear homothetic projection of factor $\lambda(p)$.


We stress that the use of the term ``projection'' is slightly abused, because $\lambda(p)$ may be zero, in which case $D F(p)=0$.
However, if $\lambda(p)\neq0$, then $D F(p)[V(G)]=V(H)$.
Since $V(H)$ is bracket-generating, $F$ is a smooth submersion on $\Omega\cap\{\lambda\neq0\}$, that is, the Lie differential $D F(p):\frk g\to\frk h$ is surjective for all $p\in\Omega\cap\{\lambda\neq0\}$.

\subsection{Gradient}
Let $G$ be a sub-Riemannian Lie group and $\Omega\subset G$ open.
The \emph{(horizontal) gradient} of a $C^1$-function $u:\Omega\to\R$ is 
$\grad_Gu:\Omega\to V(G)$ defined by
\begin{equation}\label{hgradDef1}
Du(p)[v] = \langle \grad_Gu(p),v \rangle_G,
\qquad\forall p\in\Omega,\forall v\in V(G) .
\end{equation}
If we want to see the gradient as a vector field on $G$, then we write
\[
\tilde\grad_Gu(p) := dL_p[\grad_Gu(p)] \in \tilde V(G)_p \subset T_p G .
\]

If $X_1,\dots,X_r$ is an orthonormal basis of $V(G)$, 
then, for all $u\in C^1(\Omega)$,
\begin{equation}\label{eq670d3342}
 	\grad_G u 
	= \sum_{i=1}^r (\tilde X_i u) X_i 
	= \sum_{i=1}^r Du[X_i] X_i .
\end{equation}

Let $G$ and $H$ be sub-Riemannian Lie groups and $\Omega\subset G$ open.
If $F:\Omega\to H$ is a contact map, then,
for all $p\in \Omega$, $v\in V(G)$ and $u:H\to\R$ smooth, Proposition~\ref{PropChainRule1_2} and~\eqref{hgradDef1} imply that
\begin{equation}\label{eq670e18fb}
\langle \grad_G(u\circ F)(p) , v \rangle_G = \langle \grad_Hu(F(p)),DF(p)[v] \rangle_H .
\end{equation}

The gradient is left-invariant, in the sense that
if $u\in C^1(\Omega)$, $a\in G$, 
and $v\in V(G)$,
then
\begin{equation}\label{eq67058b31}
	\grad_G(u\circ L_a) = (\grad_Gu)\circ L_a .
\end{equation}
Indeed, if $p\in a^{-1}\Omega$,
\[
	\langle \grad_G(u\circ L_a)(p) , v \rangle_G 
	\overset{\eqref{eq670e18fb}}= \langle \grad_Gu(L_a(p)),DL_a(p)[v] \rangle_G
	\overset{\ref{PropChainRule1_3}}=  \langle \grad_Gu(L_a(p)),v \rangle_G .
\]

\subsection{Haar measure and modular function}\label{sec670d69d9}
Let $G$ be a Lie group, let $\vol_G$ denote a left-invariant Haar measure on $G$. 
With a slight abuse of notation, we write $\dd\vol_G$ for the volume form that represents the measure $\vol_G$.

The measure $\vol_G$ is left invariant in the sense that
\begin{equation}\label{eq6762c376}
	\forall a\in G, \forall f\in L^1(\vol_G),
	\qquad
	\int_G f(y) \dd\vol_G(y) = \int_G f(ax)\dd\vol_G(x) .
\end{equation}
Equivalently, the volume form~$\dd\vol_G$ is left invariant, in the sense that
\[
	\forall a,x\in G
	\qquad
	dL_a|_x^* \dd\vol_G(ax) = \dd\vol_G(x) .
\]


The \emph{modular function} of a Lie group $G$ is the smooth group morphism $\mu_G:G\to (0,\infty)$ given by
\[
\mu_G(g) := \det(\Ad_g) ,
\]
where $\Ad_g:\frk g\to\frk g$ is the adjoint representation of $G$.
More explicitly, $\Ad_g = DC_g$ where $C_g(x) = gxg^{-1}$ is the conjugation by $g$.
See~\cite[Chapter 11]{MR1681462} for further details.

The role of the modular function can be described as follows.
If $g\in G$, then $R_g^*\dd\vol_G$ is again a left-invariant volume form, because left and right translations commute with each other.
Therefore, there is a constant $\mu>0$ such that $R_g^*\dd\vol_G = \mu\dd\vol_G$:
clearly, such constant depends on $g$ but not on the choice of $\dd\vol_G$.
One can compute $\mu$ and see that it is exactly $\mu_G(g)^{-1}$.
Indeed, since $C_g = L_g\circ R_{g^{-1}}$ and $\dd\vol_G$ is left invariant, we have $C_g^*\dd\vol_G = R_{g^{-1}}^*\dd\vol_G$;
on the other hand, $C_g^*\dd\vol_G = \det(\Ad_g)\dd\vol_G = \mu_G(g)\dd\vol_G$, because $C_g$ is a group automorphism fixing $1_G$ and pulling back a left-invariant volume form by an automorphism rescales it by the determinant of its differential at $1_G$.
In summary, for all $g\in G$,
\begin{equation}\label{eq6762d704}
	C_g^*\dd\vol_G = \mu_G(g) \dd\vol_G ,
	\qquad\text{and}\qquad
	R_g^*\dd\vol_G = \mu_G(g)^{-1} \dd\vol_G .
\end{equation}
Moreover, for all $g\in G$ and $f\in L^1_{\rm loc}(\vol_G)$,
\begin{equation}
	 \int_G f(xg) \dd\vol_G(x) = \mu_G(g) \int_G f(x) \dd\vol_G(x) .
\end{equation}

If $\mu_G\equiv1$, then $G$ is called \emph{unimodular}.
Left-invariant Haar measures on unimodular Lie groups are also right invariant.

\subsection{Divergence}

Let $G$ be a Lie group, let $\vol_G$ denote a left-invariant Haar measure on $G$ and let $\Omega\subset G$ be open. 

The \emph{Lie derivative} of a differential form $\omega$ on $G$ along a vector field $\tilde X\in\Gamma^1(T\Omega)$ is the continuous differential form
\[
\scr L_{\tilde X}\omega := \left.\frac{\dd}{\dd t}\right|_{t=0}	(\Phi_{\tilde X}^{t})^*\omega ,
\]
where $\Phi_{\tilde X}$ is the flow of $\tilde X$.

The \emph{divergence} of a $C^1$ vector field $\tilde X\in \Gamma^1(T\Omega)$ is 
the real-valued function $\div_G\tilde X:\Omega\to\R$
defined by the identity
\begin{equation}\label{LieDivdef}
	\scr L_{\tilde X}\dd\vol_G = (\div_G\tilde X) \dd\vol_G ,
\end{equation}
where $\scr L$ is the Lie derivative. 
Notice that the divergence does not depend on the choice of a particular Haar measure of $G$.

\begin{lemma}\label{lem6771cc90}
	Let $\Omega\subset G$ be an open subset of $G$,
	and $\tilde X\in \Gamma^1(T\Omega)$.
	A continuous function $f:\Omega\to\R$ is the divergence of $\tilde X$ with respect to $\vol_G$, i.e., $f=\div_G\tilde X$,
	if and only if, for all $\phi\in C^1_c(\Omega)$,
	\begin{equation}\label{eq6762ac39}
	\int_\Omega \phi\cdot f  \dd\vol_G
	= - \int_\Omega d\phi[\tilde X] \dd\vol_G .
	\end{equation}
\end{lemma}
\begin{proof}
	We start with the following standard identity of differential forms (see~\cite[Proposition 12.32]{MR2954043}) for $\tilde X\in \Gamma^1(T\Omega)$ and $\phi\in C^1(\Omega)$:
	\begin{equation}\label{eq6762bec9}
	\begin{aligned}
	\scr L_{\tilde X}(\phi\dd\vol_G) 
	&= d\phi[\tilde X] \cdot\dd\vol_G + \phi\cdot \scr L_{\tilde X}\dd\vol_G \\
	&= d\phi[\tilde X] \cdot\dd\vol_G + \phi\cdot (\div_G\tilde X) \dd\vol_G , 
	\end{aligned}
	\end{equation}
	where $\phi\in C^1(\Omega)$.
	The lemma follows from the fact that, for all $\phi\in C^1_c(\Omega)$
	\begin{equation}\label{eq6762ad54}
	\int_\Omega \scr L_{\tilde X}(\phi\vol_G) = 0 .
	\end{equation}
	To prove~\eqref{eq6762ad54} given $\phi\in C^1_c(\Omega)$, one takes $\psi\in C^\infty_c(\Omega)$ such that $\psi=1$ on a neighborhood of the support of $\phi$.
	Then, the vector field $\tilde Y:=\psi\tilde X$ has compact support in $\Omega$ and thus it is complete.
	Moreover, there exists $\epsilon>0$ such that, if $|t|<\epsilon$, then $\Phi_{\tilde Y}^t(p) = \Phi_{\tilde X}^t(p)$ for all $p$ in the support of $\phi$.
	Hence, for $|t|<\epsilon$, we have $(\Phi_{\tilde X}^t)^*(\phi\dd\vol_G) = (\Phi_{\tilde Y}^t)^*(\phi\dd\vol_G)$.
	Since $\Phi_{\tilde Y}^t$ is a orientation-preserving diffeomorphism $\Omega\to\Omega$, then (see~\cite[Proposition 16.6]{MR2954043})
	\[
	\int_\Omega (\Phi_{\tilde Y}^t)^*(\phi\dd\vol_G)
	= \int_\Omega \phi\dd\vol_G .
	\]
	It follows that
	\begin{align*}
		0 
		&= \left.\frac{\dd}{\dd t}\right|_{t=0} \int_\Omega (\Phi_{\tilde Y}^t)^*(\phi\dd\vol_G) 
		= \left.\frac{\dd}{\dd t}\right|_{t=0} \int_\Omega (\Phi_{\tilde X}^t)^*(\phi\dd\vol_G) \\
		&= \int_\Omega \left.\frac{\dd}{\dd t}\right|_{t=0} (\Phi_{\tilde X}^t)^*(\phi\dd\vol_G) 
		= \int_\Omega \scr L_{\tilde X}(\phi\dd\vol_G) .
	\end{align*}
	This proves~\eqref{eq6762ad54} and thus the lemma.
\end{proof}
 
If $\tilde X\in \Gamma^1(T\Omega)$ is a $C^1$ vector field and $f\in C^1(\Omega)$, then (see~\eqref{eq6762bec9} and~\cite[Proposition 12.32]{MR2954043})
\begin{equation}\label{eq670d3315}
\div_G(f\tilde X) = df[\tilde X] + f \div_G(\tilde X) .
\end{equation}

\begin{lemma}
	The divergence is left invariant, in the sense that,
	for all $\tilde X\in\Gamma^1(TG)$ and $a\in G$,
	\begin{equation}\label{eq676300e5}
	\div_G((L_a)_*\tilde X) = (\div_G\tilde X)\circ L_{a^{-1}} .
	\end{equation}
\end{lemma}
\begin{proof}
	Let $\tilde X\in\Gamma^1(TG)$ and $a\in G$.
	Then we have:
	\begin{align*}
		\int_\Omega \phi\cdot\div_G((L_a)_*\tilde X) \dd\vol_G
		&\overset{L.~\ref{lem6771cc90}}= - \int_\Omega d\phi[(L_a)_*\tilde X] \dd\vol_G \\
		&= - \int_\Omega d\phi[dL_a|_{a^{-1}x}\tilde X(a^{-1}x) ] \dd\vol_G(x) \\
		&= - \int_\Omega d(\phi\circ L_a)|_{a^{-1}x} [\tilde X(a^{-1}x)] \dd\vol_G(x) \\
		&\overset{\eqref{eq6762c376}}=  - \int_{a^{-1}\Omega} d(\phi\circ L_a)|_{y} [\tilde X(y)] \dd\vol_G(y) \\
		&\overset{L.~\ref{lem6771cc90}}= \int_{a^{-1}\Omega} (\phi\circ L_a)(y) \div_G(\tilde X)(y) \dd\vol_G(y) \\
		&\overset{\eqref{eq6762c376}}= \int_{\Omega} (\phi\circ L_a)(ax) \div_G(\tilde X)(ax) \dd\vol_G(x) \\
		&= \int_\Omega \phi(x) \div_G\tilde X(a^{-1}x) \dd\vol_G(x) .
	\end{align*}
\end{proof}

\begin{lemma}
	If $X\in\frk g$, then the divergence of the left-invariant vector field $\tilde X\in\Gamma(TG)$ is the constant
	\begin{equation}\label{eq670d32ac}
	\div_G\tilde X = - D(\log\mu_G)[X] = -\trace(\ad_X) .
	\end{equation}
	In particular, $G$ is unimodular if and only if $\div_G\tilde X= 0$ for all left-invariant vector fields $\tilde X$.
\end{lemma}
\begin{proof}
	If $X\in\frk g$, then the flow of the left-invariant vector field $\tilde X$ is
	\[
	\Phi^t_{\tilde X}(g) = g\exp(tX) = R_{\exp(tX)}(g) .
	\]
	From~\eqref{eq6762d704} and the fact that $\mu_G$ is a group morphism, it follows that
	\begin{align*}
		\left.\frac{\dd}{\dd t}\right|_{t=0} (\Phi^t_{\tilde X})^* \dd\vol_G
		&= \left.\frac{\dd}{\dd t}\right|_{t=0} R_{\exp(tX)}^* \dd\vol_G \\
		&= \left.\frac{\dd}{\dd t}\right|_{t=0} \mu_G(\exp(tX))^{-1} \dd\vol_G
		= -D\mu_G(1_G)[X] \dd\vol_G .
	\end{align*}
	So, we have, for every $p\in G$,
	\begin{equation}\label{eq6a7d6841}
		(\div_G\tilde X)(p) = -D\mu_G(1_G)[X] .
	\end{equation}
	Since $\mu_G:G\to (0,+\infty)$ is a multiplicative group morphism,
	we have 
	$D\mu_G(p)[X] 
	= \left.\frac{\dd}{\dd t}\right|_{t=0} \mu_G(p\exp(tX))
	= \left.\frac{\dd}{\dd t}\right|_{t=0} \mu_G(p)\mu_G(\exp(tX))
	= \mu_G(p) D\mu_G(1_G)[X]$.
	Therefore,~\eqref{eq6a7d6841} becomes 
	\begin{equation}
		(\div_G\tilde X)(p) = - \frac{ D \mu_G(p)[X] }{ \mu_G(p) }
		= - D(\log(\mu_G))(p)[X] ,
	\end{equation}
	and we get~\eqref{eq670d32ac}.
	
	Finally, the fact that $D\mu_G(e)[X] = \trace(\ad_X)$ follows from the Jacobi Formula for the derivative of the determinant.
	Indeed, 
	\[
	D\mu_G[X]
	= \left.\frac{\dd}{\dd t}\right|_{t=0} \det(\Ad_{\exp(tX)})
	= \left.\frac{\dd}{\dd t}\right|_{t=0} \det(e^{t\ad_X})
	\overset{\substack{\text{Jacobi}\\\text{Formula}}}= \trace(\ad_X) .
	\]
\end{proof}

\subsection{Sub-Laplacian of a sub-Riemannian Lie group}
Let $G$ be a sub-Rieman\-nian Lie group, $\Omega\subset G$ open and $\vol_G$ a fixed Haar measure on $G$.
A \emph{sub-Laplacian} on $G$ is the operator $ \laplacian_G: C^2(\Omega)  \to  C^0(\Omega)$ defined by 
\[
\laplacian_G u = \div_G(\tilde \grad_G u).
\] 
for all  $u\in C^\infty(\Omega)$. 
If $X_1,\dots,X_r$ is an orthonormal basis of $V(G)$, 
then, by~\eqref{eq67683e74}, \eqref{eq670d3342}, \eqref{eq670d3315}, and~\eqref{eq670d32ac},
\begin{equation} \label{eq670d6c55}
\begin{aligned}
	\laplacian_G u
	&= \sum_{i=1}^r (\tilde X_i^2 u + \tilde X_iu \div_G\tilde X_i ) \\
	&= \sum_{i=1}^r  D^2u[X_i,X_i] - \langle \grad_Gu,\grad_G(\log\mu_G) \rangle_G .
\end{aligned}
\end{equation}
where $\mu_G$ is the modular function.
If $G$ is unimodular, then the above formula simplifies to the sum of squares
\[
\laplacian_G u = \sum_{i=1}^r \tilde X_i^2 u.
\]

\begin{lemma}
	The sub-Laplacian is left invariant, in the sense that, for all $u\in C^2(\Omega)$ with $\Omega\subset G$ open, and all $g\in G$,
	\begin{equation}\label{eq67058fa9}
		\laplacian_G(u\circ L_g) = (\laplacian_Gu)\circ L_g .
	\end{equation}
\end{lemma}
\begin{proof}
	Fix $\Omega\subset G$ open, $u\in C^2(\Omega)$ and $g\in G$.
	Set $h=g^{-1}$.
	From~\eqref{eq67058b31}, we obtain
	\begin{equation}\label{eq676301bc}
	\begin{aligned}
		(L_h)_*(\tilde\grad_Gu)(p)
		&= dL_h \circ dL_{h^{-1}p} [(\grad_Gu)(h^{-1}p)] \\
		&\overset{\eqref{eq67058b31}}= dL_p[\grad_G(u\circ L_{h^{-1}})(p)] \\
		&= \tilde\grad_G(u\circ L_{h^{-1}})(p) .
	\end{aligned}
	\end{equation}
	Therefore, using the left invariance of the divergence, we obtain
	\begin{align*}
		(\laplacian_Gu)\circ L_g
		&= (\div_G(\tilde\grad_Gu))\circ L_g \\
		&\overset{\eqref{eq676300e5}}= \div_G((L_h)_*\tilde\grad_Gu) \\
		&\overset{\eqref{eq676301bc}}= \div_G(\tilde\grad_G(u\circ L_g)) 
		= \laplacian_G(u\circ L_g) .
	\end{align*}
\end{proof}

\section{The computation for conformal submersions}\label{sec67082cbe}

\begin{theorem}\label{thm67082cc6}
	Let $G$ and $H$ be sub-Riemannian Lie groups, and $\Omega\subset G$ open.
	Let $F:\Omega\to H$ be a $C^2$-smooth conformal submersion of factor $\lambda$ between sub-Riemannian groups, as in~\ref{sec670d0abe}.
	Define $b:\Omega\to V(H)$ as
	\[
	b(p) :=  \trace_G(D^2F(p)) - DF(p)[\grad_G(\log\mu_G)(p)] + \lambda(p)^2\grad_H(\log\mu_H)(F(p)) .
	\]
	where $\mu_G$ and $\mu_H$ are the modular functions of $G$ and $H$, respectively.
	
	Then,
	for all $u\in C^\infty(H)$ and $p\in\Omega$,
	\begin{equation}\label{eq67082930}
		\laplacian_G(u\circ F)(p)
		= \lambda(p)^2 \laplacian_Hu(F(p))
			+ \langle b(p) , \grad_Hu(F(p)) \rangle_H .
	\end{equation}	
\end{theorem}
\begin{proof}
	Fix $p\in\Omega$ and $u\in C^2(H)$.
	Let  $X_1,\dots,X_r$ be an orthonormal basis of $V(G)$
	such that
	$X_{k+1},\dots,X_r$ is a basis of $\ker(DF(p))|_{V(G)}$. 
	Since $DF(p)|_{V(G)}:V(G)\to V(H)$ is a surjective homothetic projection of factor $\lambda(p)$,
	the vectors
	$Y_1 := DF(p)X_1/\lambda(p), \dots, Y_k := DF(p)X_k/\lambda(p)$ form an orthonormal basis of $V(H)$ 
	by Proposition~\ref{prop676ff612}.
	It follows that 
	\begin{equation}\label{eq67683f08}
	\sum_{i=1}^r D^2u(p)[DF(p)X_i,DF(p)X_i]
	= \lambda(p)^2 \sum_{i=1}^k D^2u(p)[Y_i,Y_i] .
	\end{equation}
	
	We can then compute:
	\begin{align*}
		\laplacian_G(u\circ F)(p)
		&\overset{\eqref{eq670d6c55}}= \sum_{i=1}^r  D^2(u\circ F)(p)[X_i,X_i] - \langle \grad_G(u\circ F)(p) , \grad_G(\log\mu_G)(p) \rangle_G \\
		&\overset{\eqref{eq670d7a23}}= \sum_{i=1}^r D^2u(F(p))[DF(p)X_i,DF(p)X_i] \\
		&\qquad	+ \sum_{i=1}^r Du(F(p)) [ D^2F(p)[X_i,X_i] ]
			- \langle \grad_G(u\circ F)(p) , \grad_G(\log\mu_G)(p) \rangle_G \\
		&\overset{\eqref{eq67683f08}}= \lambda(p)^2 \sum_{i=1}^k D^2u(p)[Y_i,Y_i]\\
		&\qquad	+ Du(F(p)) [ \trace(D^2F(p)) ]
			- \langle \grad_G(u\circ F)(p) , \grad_G(\log\mu_G)(p) \rangle_G \\
		&\overset{\eqref{eq670d6c55}}= \lambda(p)^2 \laplacian_Hu(F(p)) \\
		&\qquad	+ \lambda(p)^2 \langle \grad_Hu(F(p)) , \grad_H(\log\mu_H)(F(p)) \rangle_H \\
		&\qquad	+ \langle \grad_Hu(F(p)),\trace(D^2F(p)) \rangle_H
			- \langle \grad_G(u\circ F)(p) , \grad_G(\log\mu_G)(p) \rangle_G .
	\end{align*}
	This concludes the proof, because
	\begin{align*}
		\langle \grad_G(u\circ F)(p) , \grad_G(\log\mu_G)(p) \rangle_G
		&= D(u\circ F)(p)[\grad_G(\log\mu_G)(p)] \\
		&\overset{\ref{PropChainRule1_2}}= Du(F(p))[DF(p)\grad_G(\log\mu_G)(p)] \\
		&= \langle \grad_Hu(F(p)),DF(p)\grad_G(\log\mu_G)(p)  \rangle_H .
	\end{align*}
\end{proof}

\section{Characterization of homothetic submersions}\label{sec670d7088}

\begin{theorem}\label{thm670e8117}
	Let $G$ and $H$ be sub-Riemannian Lie groups, with $\Omega_G\subset G$ and $\Omega_H\subset H$ open.
	Let $P_G:C^2(\Omega_G) \to C^0(\Omega_G)$ be a differential operator on $\Omega_G$, of the form
	\[
	P_Gu = a_2 \cdot \laplacian_Gu + \langle a_1,\grad_Gu \rangle_G + a_0\cdot u,
	\qquad u\in C^2(\Omega_G)
	\]
	with $a_0:\Omega_G\to\R$, $a_1:\Omega_G\to V(G)$ and $a_2:\Omega_G\to(0,\infty)$ continuous.
	Similarly, let $P_H:C^2(\Omega_H) \to C^0(\Omega_H)$ be a differential operator on $\Omega_H$, of the form
	\[
	P_Hu = b_2 \cdot \laplacian_Hu + \langle b_1,\grad_Hu \rangle_H + b_0\cdot u,
	\qquad u\in C^2(\Omega_H),
	\]
	with $b_0:\Omega_H\to\R$, $b_1:\Omega_H\to V(H)$ and $b_2:\Omega_G\to\R$ continuous.
	
	Let $F:\Omega_G\to H$ be a smooth map of class $C^2$ with $F(\Omega_G)\subset \Omega_H$ and such that, 
	for all $u\in C^2(\Omega_H)$,
	\begin{equation}\label{eq670e8147}
	P_G(u\circ F) = (P_Hu)\circ F .
	\end{equation}
	
	Then $F$ is a conformal submersion of factor $\lambda:\Omega_G\to[0,+\infty)$ with
	\[
	\lambda(p)^2 = \frac{b_2(F(p))}{a_2(p)} .
	\]
\end{theorem}

\begin{proof}
	For the proof of Theorem~\ref{thm670e8117}, we will in fact use an assumption on $F$ that is slightly weaker than~\eqref{eq670e8147}, as we describe now.

Let $G$ and $H$ be sub-Riemannian Lie groups.
Let $\Omega\subset G$ be an open subset and $F:\Omega\to H$ a smooth map.
Suppose that there are functions
$a_1:\Omega\to V_1(G)$, $a_2,a_3,f,b_2,b_3:\Omega\to\R$,
and $b_1:\Omega\to V_1(H)$, such that,
for every $u\in C^\infty(H)$,
\begin{equation}\label{eq670d7216}
\begin{split}
	&\laplacian_G(u\circ F) + \langle a_1 , \grad_G(u\circ F) \rangle_G + a_2 \cdot(u\circ F) + a_3 \\
	&\qquad= f \cdot (\laplacian_H u)\circ F + \langle b_1 , (\grad_Hu)\circ F \rangle_H + b_2 \cdot (u\circ F) + b_3 .
\end{split}
\end{equation}
The functions $f, b_1, b_2, b_3$ are taken on the domain $\Omega$, but they can be of the form $\phi\circ F$ with $\phi$ defined on $H$.
In other words, the right-hand side of~\eqref{eq670d7216} is meant to be a differential operator on $H$, generalizing~\eqref{eq670e8147}.

Since~\eqref{eq670d7216} must holds for constant functions, we automatically must have that  $a_2 \cdot(u\circ F) + a_3 = b_2 \cdot (u\circ F) + b_3$ and \eqref{eq670d7216} implies that
\begin{equation}\label{eq670e195c}
	\laplacian_G(u\circ F) + \langle a_1 , \grad_G(u\circ F) \rangle_G 
	= f \cdot (\laplacian_H u)\circ F + \langle b_1 , (\grad_Hu)\circ F \rangle_H  ,
\end{equation}
for every $u\in C^\infty(H)$.

Using~\eqref{eq670e18fb}, we have that
$\langle a_1 , \grad_G(u\circ F) \rangle_G = \langle DF[a_1],(\grad_Hu)\circ F \rangle_H$
and thus~\eqref{eq670e195c} becomes
\begin{equation}\label{eq66f481ba}
	\laplacian_G(u\circ F) 
	= f \cdot (\laplacian_H u)\circ F + \langle b , (\grad_Hu)\circ F \rangle_H ,
\end{equation}
with $b:\Omega\to V(H)$ given by $b(p) := b_1(p) - DF(p)[a_1(p)]$.

We shall show that~\eqref{eq66f481ba} implies that $F$ is a conformal submersion of factor $\lambda = \sqrt{f}$ and thus the function $b$ in~\eqref{eq66f481ba} has the form given by Theorem~\ref{thm67082cc6}.

Let $U\subset H$ be an open neighborhood of $1_H$ such that the map $\log:U\to\frk h$, i.e., the inverse of the exponential map $\exp$ of $H$, is well defined.
Let $\phi\in C^\infty_c(H)$ be a bump function such that $1_H$ is in the interior of $\{\phi=1\}$ and the support of $\phi$ is contained in $U$.

We denote by $\frk h^*$ the dual space of $\frk h$ and by $\langle \cdot|\cdot \rangle:\frk h\times\frk h^*\to\R$ the standard pairing contraction.
For $\alpha\in\frk h^*$ and $\hat q\in H$,
define $u_{\alpha}, u_{\alpha}^{\hat q}\in C^\infty_c(H)$ by
\begin{equation}\label{eq66f51811}
	u_{\alpha}(q) := \phi(q) \cdot \langle \alpha | \log(q) \rangle^2 ,
	\qquad\text{ and }\qquad
	u_{\alpha}^{\hat q}(q) := u_{\alpha}(\hat q^{-1} q) .
\end{equation}
Notice that, if $X\in\frk h$ and $t\in\R$ are such that $\phi(\exp(tX))=1$, then
\[
u_{\alpha}(\exp(tX))
= t^2 \langle \alpha | X \rangle^2 .
\]
If $v,w\in V(H)$, then 
\begin{equation}\label{eq6773d391}
\begin{aligned}
Du_{\alpha}^{\hat q}(\hat q)[v] 
&= \left. \frac{\dd}{\dd t} \right|_{t=0} u_{\alpha}^{\hat q} (\hat q\exp(tv))
= \left. \frac{\dd}{\dd t} \right|_{t=0} t^2 \langle \alpha | v \rangle^2
= 0 , \quad\text{ and } \\
D^2u_{\alpha}^{\hat q}(\hat q)[v,v] 
&= \left. \frac{\dd^2}{\dd t^2} \right|_{t=0} u_{\alpha}^{\hat q} (\hat q\exp(tv))
= \left. \frac{\dd^2}{\dd t^2} \right|_{t=0} t^2 \langle \alpha | v \rangle^2
= 2 \langle \alpha | v \rangle^2 .
\end{aligned}
\end{equation}

Fix an orthonormal basis $Y_1,\dots,Y_s$ of $V_1(H)$.
From~\eqref{eq6773d391} and~\eqref{eq670d6c55}, we obtain 
\begin{equation}\label{eq6773d56f}
\grad_H u_{\alpha}^{\hat q}(\hat q) = 0 
\qquad\text{and}\qquad
\laplacian_H u_{\alpha}^{\hat q}(\hat q) = \sum_{i=1}^{s} 2 \langle \alpha | Y_i \rangle^2 .
\end{equation}

Fix an orthonormal basis $X_1,\dots,X_r$ of $V_1(G)$,
and consider a function $F:\Omega\to H$ and one point $\hat p\in\Omega$ with $F(\hat p) = \hat q$.  
The sub-Laplacian of $u_{\alpha}^{\hat q}\circ F$ at $\hat p$ is
\begin{equation}\label{eq6773d547}
\begin{aligned}
	\laplacian_G (u_{\alpha}^{\hat q}\circ F)(\hat p)
	&\overset{\eqref{eq670d6c55}}= \sum_{i=1}^r D^2(u_{\alpha}^{\hat q}\circ F)(\hat p)[X_i,X_i] - D(u_{\alpha}^{\hat q}\circ F)(\hat p)[\grad_G(\log\mu_G)] \\
	&\hspace{-1em}\overset{\eqref{eq670d7a23}\&\ref{PropChainRule1_2}}{=} \sum_{i=1}^r
	D^2u_{\alpha}^{\hat q}(F(\hat p)) [ DF(\hat p)X_i , DF(\hat p)X_i ] \\
	&\qquad+ Du_{\alpha}^{\hat q}(F(\hat p))[D^2F(\hat p)[X_i,X_i]] \\
	&\qquad\qquad- Du_{\alpha}^{\hat q}( F(\hat p))[DF(\hat p)[\grad_G(\log\mu_G)] ]\\
	&\overset{\eqref{eq6773d391}}= \sum_{i=1}^r 2\langle \alpha | DF(\hat p)X_i \rangle^2 .
\end{aligned}
\end{equation}

Now, we apply~\eqref{eq6773d56f} and~\eqref{eq6773d547} to~\eqref{eq66f481ba}, and we obtain
\begin{equation}\label{eq66f52952}
\sum_{i=1}^{r} \langle \alpha | DF(\hat p)X_i \rangle^2
=
f(\hat p) \sum_{i=1}^{s} \langle \alpha | Y_i \rangle^2 ,
\qquad\forall\alpha\in\frk h^* .
\end{equation}
First of all,
the identity~\eqref{eq66f52952} implies that $DF(\hat p)[V(G)]\subset V(H)$.
Indeed, if $\alpha\in\frk h^*$ is such that $V(H)\subset\ker(\alpha)$, then~\eqref{eq66f52952} implies $\langle \alpha | DF(\hat p)X_i \rangle^2=0$ for all $i$ and thus $DF(\hat p)[V(G)]\subset\ker(\alpha)$.

We have now that the map $DF(\hat p)|_{V(G)}:V(G)\to V(H)$ is a linear map.
If $DF(\hat p)|_{V(G)}=0$, then $f(\hat p)=0$;
otherwise, if $DF(\hat p)|_{V(G)}\neq0$, then $f(\hat p)>0$ and $DF(\hat p)|_{V(G)}$ is a homothetic projection of factor $\lambda(\hat p) = \sqrt{f(p)}$,
using Proposition~\ref{prop676ff612}.\ref{prop676ff612_4}.

Since $\hat p$ is an arbitrary point in $\Omega$, we conclude that $F$ is a conformal submersion of factor $\lambda$.
\end{proof}

\section{Consequences and further results}

\subsection{Equivalent sums of squares}
	
\begin{corollary}\label{cor676308cb}
	Let $G$ be a polarized Lie group.
	Let $X=(X_1,\dots,X_r)$ and $Y=(Y_1,\dots,Y_r)$ be two bases of $V(G)$.
	Define the differential operators
	\[
	P_X := \sum_{i=1}^r \tilde X_i^2 + \div_G(\tilde X_i) \tilde X_i
	\quad\text{ and }\quad
	P_Y := \sum_{i=1}^r \tilde Y_i^2 + \div_G(\tilde Y_i) \tilde Y_i .
	\]
	Then the following statements are equivalent:
	\begin{enumerate}[label=(\roman*)]
	\item\label{cor676308cb_1}
		$P_X=P_Y$;
	\item\label{cor676308cb_2}
		There exists a scalar product on $V(G)$ for which both $X$ and $Y$ are orthonormal bases;
	\item\label{cor676308cb_3}
		There exists an $r\times r$ invertible matrix $A$ with $A^{-1} = A^T$ such that $Y_i = \sum_j A_i^jX_j$ for all $i$.
	\end{enumerate}
\end{corollary}
\begin{proof}
	It is clear that the conditions \ref{cor676308cb_2} and \ref{cor676308cb_3} are equivalent.
	
	$\ref{cor676308cb_2}\THEN\ref{cor676308cb_1}$
	If $X$ and $Y$ are orthonormal bases for the same (now fixed) scalar product, then $P_X$ and $P_Y$ are just the same sub-Laplacian $\laplacian_G$ for the sub-Riemannian Lie group $G$.
	
	$\ref{cor676308cb_1}\THEN\ref{cor676308cb_2}$
	Consider the two sub-Riemannian Lie groups $G_X$ and $G_Y$ that are equal to $G$ as polarized Lie group but with the scalar products defined by imposing $X$ or $Y$ orthonormal.
	Then $P_X = \laplacian_{G_X}$ and $P_Y=\laplacian_{G_Y}$.
	Theorem~\ref{thm6763089e} says that, if $P_X=P_Y$, then the identity map $G\to G$ is an isometry.
	Thus, $G_X=G_Y$.
\end{proof}

\begin{remark}
	Notice that, in Corollary~\ref{cor676308cb}, the identity $P_X=P_Y$ does not imply that there exists a transformation of $G$ that maps the basis $X$ to the basis $Y$.
	In particular, the linear transformation $A$ acts only on the space $V(G)$, but neither on $\frk g$, nor on $G$.
\end{remark}

\subsection{Carnot groups and proof of Theorem~\ref{thm670e8a84}}\label{sec67633de5}

A \emph{stratification} of a Lie algebra $\frk g$ is a splitting $\frk g = \bigoplus_{k=1}^sV_k$ such that $[V_1,V_k]=V_{k+1}$ for all $i\in\{1,\dots,s-1\}$ and $[V_1,V_s]=\{0\}$.
A \emph{Carnot group} is a connected simply connected sub-Riemannian Lie group $G$ whose polarization $V(G)\subset\frk g$ is the first layer of a stratification of $\frk g$, the Lie algebra of $G$.
See~\cite{MR3742567,donne2024metricliegroupscarnotcaratheodory} for details.

The important feature of Carnot groups is that they support dilations,
 that is, homotheties for all positive factors $\lambda>0$.
Indeed, for each $\lambda>0$, define $\delta_\lambda:G\to G$ as the Lie group automorphism such that $(\delta_\lambda)_*:\frk g\to \frk g$ is the Lie algebra automorphism given by $(\delta_\lambda)_*v=\lambda^kv$ for all $v\in V_k$.

Homotheties between Carnot groups of the same dimension are in fact isomorphisms, see \cite{MR3441517} and the proof of Proposition~\ref{prop6767ac89} below.
We don't know a corresponding characterization of conformal submersions, but we
 can characterize the pairs of Carnot groups between which there exist a conformal submersion:
the range group must be a quotient of the domain.

\begin{proposition}\label{prop6767ac89}
	Let $G$ and $H$ be sub-Riemannian Carnot groups, $\Omega\subset G$ open.
	There exists a conformal submersion $F:\Omega\to H$ of class $C^1$, if and only if there exists a Carnot morphism $F_0:G\to H$ that is a conformal submersion.
	In particular, $H$ is a quotient of $G$ by a dilation invariant normal subgroup.
	
	If $\dim(G)=\dim(H)$ and if $F$ is a conformal submersion of constant factor $\lambda>0$, then, up to precomposing with a left translation, $F$ itself is an isomorphism of Carnot groups.
\end{proposition}
\begin{proof}
	Suppose we have a $C^1$ conformal submersion $F:\Omega\to H$ of factor $\lambda:\Omega\to(0,+\infty)$.
	Since $F$ is a contact map of class $C^1$, it is Pansu differentiable at every point,
	see for instance \cite{donne2024metricliegroupscarnotcaratheodory}.
	The Pansu differential of $F$ at a point $p$ is the Lie group morphism $F_0:G\to H$ such that 
	\[
	DF_0(1_G)[v] = DF(p)[v]
	\]
	for all $v\in V_1(G)$.
	In particular, $F_0$ is also a conformal submersion of factor $\lambda(p)$.
	
	If $\dim(G)=\dim(H)$ and if $F$ is a conformal submersion of constant factor $\lambda>0$, then $\delta^H_{1/\lambda}\circ F:\Omega\to H$ is an isometry, and thus it is the restriction of an isomorphism of Carnot groups by \cite{MR3441517}.
\end{proof}

The case when $G=H$, we have the following characterization of homotheties and isometries in terms of the sub-Laplacian.
\begin{proposition}\label{prop67682013}
	If $G$ is a Carnot group, $\Omega\subset G$ open and $\lambda>0$.
	Let $F:\Omega\to G$ be a map of class $C^2$.
	The following statements are equivalent:
	\begin{enumerate}[label=(\roman*)]
	\item\label{cor6763f920_1}
	For every $u\in C^2(\Omega)$, we have
	\[
	\laplacian_G(u\circ F) = \lambda^2 (\laplacian_Gu)\circ F + \langle \trace(D^2F), (\grad_Gu)\circ F \rangle.
	\]
	\item\label{cor6763f920_2}
	There are $p\in G$ and a Lie automorphism $A\in\Aut(G)$ that is an isometry such that $F=\delta_\lambda\circ A\circ L_p$.
	\end{enumerate}
	
	In both cases, we have in fact $\trace(D^2F)=0$.
\end{proposition}
\begin{proof}
	Since $G$ is unimodular, the vector field $b$ defined in Theorem~\ref{thm6763089e} is $b=\trace(D^2F)$.
	
	$\ref{cor6763f920_1}\THEN\ref{cor6763f920_2}$.
	From Theorem~\ref{thm6763089e} we obtain that $F$ is a homothety of factor $\lambda$.
	In other words, $\delta_\lambda^{-1}\circ F$ is an isometry.
	Since isometries of Carnot groups are of the form $A\circ L_p$ for some isometry $A\in\Aut(G)$ and $p\in G$ (see~\cite{MR3441517,MR3646026}) then~\ref{cor6763f920_2} follows.
	
	$\ref{cor6763f920_2}\THEN\ref{cor6763f920_1}$.
	Functions of the form described in \ref{cor6763f920_2} are clearly conformal submersions of factor $\lambda$.
	From Theorem~\ref{thm6763089e} then follows~\ref{cor6763f920_1}.
	Since both $\delta_\lambda$ and $A$ are automorphisms, then $D^2F=0$.
	We conclude that the vector field $b$ defined in Theorem~\ref{thm6763089e} is zero.
\end{proof}

\subsection{Symplectic preparation to the Heisenberg group}\label{subs6768707b}

\newcommand{\GL}{\mathrm{GL}}

The following discussion comes from~\cite[Section 2.4]{MR3674984}.

Let $V$ be a vector space of dimension $2n$.
A \emph{symplectic form on $V$} is a bilinear map $\omega:V\times V\to\R$ that is alternating and non-singular.
For a symplectic form $\omega$ on $V$ and a scalar product $g$ on $V$, 
define the operator $A_{\omega,g}\in\GL(V)$ by
\[
\omega(v,w) = g(v,A_{\omega,g}w)
\qquad
\forall v,w\in V.
\]

The \emph{symplectic spectrum of $g$ with respect to $\omega$} is the $n$-tuple
\[
\vec r_\omega(g) = (r_1,\dots,r_n)\in (0,+\infty)^n
\]
such that $r_1\le r_2\le\dots\le r_n$ and such that $-r_1^4, \dots, -r_n^4$ are the eigenvalues of the operator $A_{\omega,g}^2$.

The symplectic spectrum is well defined, because $A_{\omega,g}^2$ has at most $n$ distinct eigenvalues and they are negative.
Indeed, notice that $A_{\omega,g}$ is $g$-antisymmetric, that is,
$g(v,A_{\omega,g}w) = - g(A_{\omega,g}v,w) $.
It follows that the eigenvalues of $A_{\omega,g}$ are purely imaginary and thus
 $A_{\omega,g}^2$ is $g$-symmetric, that is, 
$g(A_{\omega,g}^2v,w) = g(v,A_{\omega,g}^2w)$,
and it has negative eigenvalues.

It is a standard fact that any pair of symplectic forms $\omega_1,\omega_2$ on $V$ are equivalent, that is, there exists $\Psi\in\GL(V)$ such that $\omega_2=\Psi^*\omega_1$,
see for instance \cite[Exercise 2.1.15]{MR3674984}.
The next lemma gives a complete description of equivalent classes of pairs $(\omega,g)$ consisting of a symplectic form and scalar product.

\begin{lemma}[{\cite[Lemma 2.4.6]{MR3674984}}]\label{lem6768583e}
	Let $V$ be a vector space, $\omega_1,\omega_2$ symplectic forms on $V$ and $g_1,g_2$ scalar products on $V$.
	The following statements are equivalent:
	\begin{enumerate}[label=(\roman*)]
	\item\label{lem6768583e_1}
	there exists a linear isomorphism $\Psi\in\GL(V)$ such that $\omega_2=\Psi^*\omega_1$ and $g_2=\Psi^*g_1$;
	\item\label{lem6768583e_2}
	$\vec r_{\omega_1}(g_1) = \vec r_{\omega_2}(g_2)$.
	\end{enumerate}
\end{lemma}
\begin{proof}
	$\ref{lem6768583e_1}\THEN\ref{lem6768583e_2}$:
	Set $A_j:= A_{\omega_j,g_j}$.
	Then, for all $v,w\in V$,
	\begin{align*}
	g_1(\Psi v,\Psi A_2w)
	&\overset{g_2=\Psi^*g_1}= g_2(v,A_2w)
	= \omega_2(v,w) \\
	&\overset{\omega_2=\Psi^*\omega_1}= \omega_1(\Psi v,\Psi w)
	= g_1(\Psi v,A_1\Psi w) .
	\end{align*}
	It follows that $\Psi A_2 w= A_1\Psi w$ for all $w\in V$, i.e., $A_2 = \Psi^{-1} A_1\Psi$.
	Therefore, $A_1^2$ and $A_2^2$ have the same spectrum and thus $\vec r_{\omega_1}(g_1) = \vec r_{\omega_2}(g_2)$.
	
	$\ref{lem6768583e_2}\THEN\ref{lem6768583e_1}$:
	The statement follows from the following general observation.
	If $\omega$ is a symplectic form, $g$ a scalar product and $A=A_{\omega,g}$,
	then, since $A$ is $g$-skew-symmetric, there exists a basis $\scr B = (X_1,\dots,X_n,Y_1,\dots,Y_n)$ of $V$ that is orthonormal for $g$ such that  
	$AX_i = Y_i$
	and $AY_i = -r_i^2 X_i$.
	On this basis, $\omega$ is determined by the relations $\omega(v,w) = g(v,Aw)$ for $v,w\in\scr B$.
	Therefore, in the hypothesis of \ref{lem6768583e_2}, if we take this type of bases for both pairs $(\omega_j,g_j)$, then map that exchanges the bases is the required $\Psi$.
\end{proof}

\subsection{Example: Heisenberg groups}\label{subs676e6b89}
For a vector space $V$ of dimension $2n$ and a symplectic form $\omega$ on $V$, define $\heis(\omega)$ as the vector space $V\times\R$ with Lie brackets
\[
[ (v,s), (w,t) ] = (0,\omega(v,w)) .
\]
The Lie algebra $\heis(\omega)$ is a step-two stratified Lie algebra with layers $V_1=V$ and $V_2=\R$.

We denote by $\Heis(\omega)$ the corresponding connected and simply connected Lie group.
Using nilpotency and the BCH formula, we describe $\Heis(\omega)$ as the manifold $V\times\R$ with group operation
\[
(v,s)*(w,t) = \left( v+w, s+t+\frac12\omega(v,w) \right) .
\]

A linear map $\hat\Psi:\heis(\omega_2)\to\heis(\omega_1)$ is an isomorphism of stratified Lie algebras if and only if there are $\Psi\in\GL(V)$ and $\zeta_\Psi\in\R\setminus\{0\}$
such that $\hat\Psi(v,t) = (\Psi v, \zeta_\Psi t)$, and
$\Psi^*\omega_1 = \zeta_\Psi\omega_2$.
Since symplectic forms are equivalent, every two $\heis(\omega_2)$ and $\heis(\omega_1)$ of same dimension are isomorphic as stratified Lie algebras for every pair of symplectic forms $\omega_1$ and $\omega_2$.

Given a symplectic form $\omega$ and a scalar product $g$ on $V$, we define the \emph{Heisenberg group}
$\Heis(\omega,g)$ as the Carnot group with Lie algebra $\heis(\omega)$, polarization $V$ and scalar product $g$ on $V$.

\begin{theorem}\label{thm6768768f}
	Two Heisenberg groups $\Heis(\omega_1,g_1)$ and $\Heis(\omega_2,g_2)$ are isometric if and only if there exists $\rho>0$ such that
	\begin{equation}\label{eq67687127}
	\vec r_{\omega_1}(g_1) = \rho \vec r_{\omega_2}(g_2) .
	\end{equation}
	In particular, there exist $F:\Heis(\omega_2,g_2)\to \Heis(\omega_1,g_1)$ and $\lambda>0$ such that 
	\[
	\laplacian_{\Heis(\omega_2,g_2)}(u\circ F) = \lambda^2 (\laplacian_{\Heis(\omega_1,g_1)} u)\circ F,
	\qquad\forall u\in C^2(\Heis(\omega_1,g_1)),
	\]
	if and only if~\eqref{eq67687127} holds.
\end{theorem}
\begin{proof}
	Suppose that $\Heis(\omega_1,g_1)$ and $\Heis(\omega_2,g_2)$ are isometric.
	Recall from \cite{MR3646026} that isometries of nilpotent Lie groups are affine.
	Thus, there is $\Psi\in\GL(V)$ such that $\Psi^*g_1=g_2$, $\Psi^*\omega_1 = \zeta_\Psi\omega_2$ for some $\zeta_\Psi\in\R\setminus\{0\}$.
	By Lemma~\ref{lem6768583e}, we have
	\[
	\vec r_{\omega_1}(g_1) = \vec r_{\zeta_\Psi\omega_2}(g_2) .
	\]
	It is an easy observation that, following the notation from Section~\ref{subs6768707b},
	$A_{\zeta_\Psi \omega_2,g_2} = \zeta_\Psi A_{\omega_2,g_2}$.
	It follows that
	\[
	\vec r_{\zeta_\Psi\omega_2}(g_2) = \sqrt{|\zeta_\Psi|} \vec r_{\omega_2}(g_2) ,
	\]
	which shows~\eqref{eq67687127} with $\rho=\sqrt{|\zeta_\Psi|}>0$.
	
	Suppose now that~\eqref{eq67687127} holds.
	We apply Lemma~\ref{lem6768583e} again to obtain $\Psi\in\GL(V)$ such that 
	$\Psi^*g_1=g_2$ and $\Psi^*\omega_1 = \rho^2\omega_2$.
	Then the map $\hat\Psi(v,t) = (\Psi v,\rho^2 t)$ is a Lie algebra automorphism $\heis(\omega_2)\to\heis(\omega_1)$ which defines a Lie group automorphism $\Heis(\omega_2,g_2)\to\Heis(\omega_1,g_1)$ that is an isometry.
	
	The last part of the statement follows directly from Theorem~\ref{thm670e8a84}.
\end{proof}

\subsection{Example: Heisenberg groups in coordinates}
It might be interesting to have a representation in coordinates of the Heisenberg groups and their sub-Laplacians.
Consider a basis $X_1,\dots,X_n,Y_1,\dots,Y_n$ of a vector space $V$ and define the symplectic form $\omega$ by setting $\omega(X_i,Y_i)=1=-\omega(Y_i,X_i)$ for all $i$, and zero otherwise.
For $\bar r \in(0,+\infty)^n$ with
$0<r_1\le r_2\le\dots\le r_n<\infty$, define the scalar product $g_{\bar r}$ such that 
$X_1,\dots,X_n,Y_1,\dots,Y_n$ is an orthogonal basis with
\[
g_{\bar r}(X_i,X_i) = g_{\bar r}(Y_i,Y_i) = \frac{1}{r_i^2} .
\]
Then, one can easily check that $\vec r_\omega(g_{\bar r}) = \bar r$.
By Theorem~\ref{thm6768768f}, if we assume $r_1=1$, then different vectors $\bar r$ give non-isometric Carnot structures.

Using the BCH formula, we can reconstruct the Lie group $\Heis(\omega)$ as $\R^{2n+1}$ with group operation
\[
(x,y,z) * (\bar x,\bar y,\bar z)
= \left(
x+\bar x, y+\bar y, z+\bar z + \frac12 \left(\sum_{i=1}^n (x_i\bar y_i-\bar x_iy_i \right) 
\right) ,
\]
where $x,y,\bar x,\bar y\in\R^n$ and $z,\bar z\in\R$.

The left invariant vector fields are
\begin{align*}
	\tilde X_i &= \frac{\de}{\de x_i} - \frac{y_i}{2} \frac{\de}{\de z} , &
	\tilde Y_i &= \frac{\de}{\de y_i} + \frac{x_i}{2} \frac{\de}{\de z} , &
	\tilde Z &= \frac{\de}{\de z} .
\end{align*}

An orthonormal basis of $V$ for $g_{\bar r}$ is
\begin{align*}
	\Xi_i &= r_iX_i \qquad\text{ if }i\in\{1,\dots,n\} , \\
	\Xi_{n+i} &= r_iY_i \qquad\text{ if }i\in\{1,\dots,n\} .
\end{align*}
We can express the sub-Laplacian $\laplacian_{\bar r}$ of $\Heis(\omega,g_{\bar r})$ with formula~\eqref{eq670d6c55} as
\begin{align*}
	\laplacian_{\bar r} = \sum_{i=1}^{2n} \tilde \Xi_i^2
	&= \sum_{i=1}^n r_i^2 (\tilde X_i^2 + \tilde Y_i^2) \\
	&= \sum_{i=1}^n r_i^2 \left( \frac{\de^2}{\de x_i^2} + \frac{\de^2}{\de y_i^2} + \frac{x_i^2+y_i^y}{4} \frac{\de^2}{\de z^2} + \left( x_i\frac{\de}{\de y_i} - y_i \frac{\de}{\de x_i} \right) \right).
\end{align*}
In the latter expression, different choices of $\bar r$ with $r_1=1$, give non-equivalent sub-Laplacians on $\Heis(\omega)\simeq\R^{2n+1}$.
We have thus a simple explanation of the phenomena described in \cite[\S16.3.1]{MR2363343}.

\printbibliography
\end{document}